\documentclass[12pt]{amsart}
\usepackage{graphicx}
\usepackage{geometry}
\usepackage{amssymb, amsmath, amsfonts}
\usepackage{url}
\usepackage[pdfborder={0 0 0}]{hyperref}
\usepackage{verbatim} 
\usepackage[all]{xy}
\usepackage{color}

\usepackage{marginnote}

\usepackage{breakurl}

\allowdisplaybreaks
\newtheorem{theorem}{Theorem}[section]

\newtheorem{lemma}[theorem]{Lemma}

\numberwithin{equation}{section}

\newif\ifcolor
\colortrue  


\begin{document}
	
	\raggedbottom
	
	\title{H\"older continuity for the $p$-Laplace equation using a differential inequality }
	
	\author[F. H\o eg]{Fredrik Arbo H\o eg}
	\address{Department of Mathematical Sciences\\
		Norwegian University of Science and Technology\\
		NO-7491 Trondheim\\ Norway}
	\email{fredrik.hoeg@ntnu.no}

	\maketitle
	
	
	\begin{abstract}
		We study H\"older continuity for solutions of the $p$-Laplace equation. This is established through a method involving an ordinary differential inequality, in contrast to the classical proof of the De Giorgi-Nash-Moser Theorem which uses iterations of an inequality through concentric balls.
	\end{abstract}

	\section{Introduction}
In this paper, we study  solutions of the $p$-Laplace equation, 
\begin{equation}
\Delta_p u =\text{div}\left( |\nabla u|^{p-2} \nabla u  \right)=0, \quad 1 < p < \infty.
\label{plaplace}
\end{equation}
It is a nonlinear second order partial differential equation in divergence form. The equation is \textit{degenerate} for $p>2$ and \textit{singular} for $p<2$.

Due to the singularity and nonlinearity of the $p$-Laplace operator, we cannot always expect solutions of equation \eqref{plaplace} to be smooth. We say that $u$ is a weak solution of equation \eqref{plaplace} in a domain $\Omega \subset \mathbb{R}^n$ provided
\begin{align*}
\int_{\Omega} |\nabla u|^{p-2}\left \langle  \nabla u, \nabla \phi \right \rangle \, dx=0
\end{align*}
for all smooth test functions $\phi$. 

We will show H\"older continuity for weak solutions of the $p$-Laplace equation. This is not a new result, see for example [LU]. We shall replace the iteration in
DeGiorgi's method by a single differential inequality. Such a device was
used for linear equations by P. Tilli in [T] . We shall apply the same
strategy here.

Throughout the text, we will restrict the range of values for $p$ to $$1<p<n.$$ 
For $p > n$, in fact all functions in $W^{1,p}_\text{loc} (\Omega)$ are continuous, see Theorem 7.17 in [GT]. For $p=n$, the proof is based on Morrey's Lemma, which can be found in Theorem 7.19 in [GT]. 

A crucial step in proving H\"older continuity is the following \textit{Oscillation Theorem.} If the Oscillation Theorem is true, the H\"older continuity follows as in page 3 in [T].

\begin{theorem}\textbf{The Oscillation Theorem}
	
	Let $u$ be a bounded weak solution of 
	$$\Delta_p u =0 \quad \text{in} \, B_4.$$
	If $$|\{ u \leq 0 \} \cap B_1| \geq \frac{1}{2}|B_1|, $$ then
	$$\sup_{B_1} u^+ \leq C |\{   u>0 \} \cap B_2|^{\alpha} \sup_{B_4} u^+$$
	\label{osclemma}
	for some $\alpha \in (0,1)$.
\end{theorem}

The paper is organized as follows. First, we list some preliminaries for the study. Then, we focus our attention on the Oscillation Theorem. To prove this, we derive an ordinary differential inequality which is used along with a modified Caccioppoli inequality. This is in contrast to the classical proofs in for example [Dg].

\section{Preliminaries}

%

\textbf{Caccioppoli type inequalities.} We give some inequalities needed later for solutions of the $p$-Laplace equation. 

\begin{lemma}\textbf{Caccioppoli inequality:}
	Let $u$ be a weak solution of \eqref{plaplace} in a domain $\Omega$ and let $r>0$ with $B_{2r} \subset \Omega$. Then
	\begin{align*}
	\int_{B_r} |\nabla u|^p \, dx \leq p^p r^{-p} \int_{B_{2r}} |u|^p \, dx. 
	\end{align*}
\label{cacc1}
\end{lemma}

\begin{proof}
Use the test function $\phi = u \zeta^p$, where $\zeta$ is a suitable
function with support in $B_{2r}$.

\end{proof}

We use  a variant of the Caccioppoli inequality, where the ball on the
left-hand side has \emph{the same radius} as the sphere on the
right-hand side. In the next Lemma, $(u-k)^+ = \text{max}\left \{ (u-k),0\right\}$. 

\begin{lemma}
Assume $u$ is a weak solution of equation \eqref{plaplace} in $B_R$. Then
\begin{align*}
\int_{B_r} \left | \nabla (u-k)^+  \right|^{p} \, dx \leq \int_{\partial B_r} \left | \nabla (u-k)^+   \right|^{p-1} (u-k)^+ \, ds.
\end{align*}
for almost every $r<R$. 
\label{cacc2}
\end{lemma}

\begin{proof}
For a formal proof, multiply equation \eqref{plaplace} by $(u-k)^+$, and use the divergence theorem. More carefully, we multiply equation \eqref{plaplace} with $\eta (u-k)^+$ where $\eta \in H_0^1 (B_r)$. This gives

\begin{align*}
0&= \int_{B_r} \eta (u-k)^+ \text{div} \left(  |\nabla u|^{p-2} \nabla u \right) \, dx \\
&= -\int_{B_r} \eta  |\nabla u|^{p-2} \left \langle  \nabla (u-k)^+ , \nabla u  \right \rangle \, dx  - \int_{B_r} (u-k)^+ |\nabla u|^{p-2} \left \langle  \nabla u, \nabla \eta \right \rangle \, dx,
\end{align*}
\begin{align*}
\int_{B_r} \eta |\nabla (u-k)^+ |^{p} \, dx \leq \int_{B_r} (u-k)^+ |\nabla (u-k)^+|^{p-1} |\nabla \eta| \, dx.
\end{align*}
Taking $$\eta= \eta_{\epsilon }(x) = \min\{1, \frac{r-|x|}{\epsilon} \}$$
and letting $\epsilon\rightarrow 0$ proves the Lemma. 

\end{proof}

\textbf{A modified Poincar\'e inequality.}
The following can be found in Theorem 3.16 in [G].

\begin{theorem}
Let $\Omega \subset \mathbb{R}^n$ be open and connected with a Lipschitz continuous boundary. For every $u \in W^{1,p}(\Omega)$, $p<n$, taking the value zero in a set $A$ with positive measure, we have
\begin{align*}
||u||_{p^*} \leq c \left( \frac{|\Omega|}{|A|}  \right)^{\frac{1}{p^*}}||\nabla u||_p,
\end{align*}
where $p^* = \frac{np}{n-p}$.
\label{poincarezero}
\end{theorem}

\section{A differential equation}
Let 
\begin{align*}
g(\rho) = \int_{B_{2-\rho}} \left| \nabla (u-k\rho)^+ \right|^{p-1} (u-k\rho)^+ \, dx, \quad \rho \in (0,1).
\end{align*}
Our plan is to find a differential equation for $g$ and to eventually show that $g(1)=0$. If we can show that $g(1)=0$, the assumptions in Theorem \ref{osclemma} can be applied to get the desired bound.

Differentiating we find
\begin{align*}
-g'(\rho) &= \int_{\partial B_{2-\rho}} \left|  \nabla (u-k\rho)^+ \right |^{p-1} (u-k\rho)^+ \, dx  + k \int_{B_{2-\rho}} \left| \nabla (u-k\rho)^+  \right|^{p-1} \, dx \\
& \equiv a(\rho) + k b(\rho).
\end{align*}
For the differentiation, we rewrote $g$ as 
\begin{align*}
g(\rho) = \int_{B_{2-\rho}} \left| \nabla (u-k\rho)^+ \right|^{p-1} (u-k\rho)^+ \, dx = \int_{B_{2-\rho}} |\nabla u |^{p-1} (u-k\rho)^+ \, dx. 
\end{align*}

We want to connect $g'$ and $g$ to get an ordinary differential equation. By H\"older's inequality and the modified Caccioppoli inequality, Lemma \ref{cacc2}, we find
\begin{align*}
g^p & \leq  \left \{  \int_{B_{2-\rho}} \left |   \nabla (u-k\rho)^+    \right|^p \, dx    \right \}^{p-1} \int_{B_{2-\rho}} \left( (u-k\rho)^+ \right)^p \, dx \\
& \leq  \left\{ \int_{\partial B_{2-\rho}}  \left| \nabla (u-k\rho)^+ \right|^{p-1} (u-k\rho)^+  \, ds \right\}^{p-1} \int_{B_{2-\rho}} \left( (u-k\rho)^+ \right)^p \, dx \\
&=\left( a(\rho )\right)^{p-1} \int_{B_{2-\rho}} \left( (u-k\rho)^+ \right)^p \, dx.
\end{align*}

Let $$M=\sup_{B_4} u^+$$ and recall the Sobolev conjugate $$p^*=\frac{np}{n-p}.$$ Under the assumptions in the Oscillation Theorem, we apply Theorem \ref{poincarezero} to get
\begin{align*}
g^p &\leq a^{p-1} \int_{B_{2-\rho}} \left( (u-k\rho )^+   \right)^{p} \, dx \\
& \leq a^{p-1} M^{\frac{n+p-p^2 }{n-(p-1)}} \int_{B_{2-\rho}} \left( (u-k\rho)^+ \right)^{(p-1)^*} \, dx \\
& \leq CM^{\frac{n+p-p^2 }{n-(p-1)}}a^{p-1} \left\{  \int_{B_{2-\rho}} \left|  \left( \nabla (u-k\rho)^+ \right)     \right|^{p-1} \, dx \right \}^{\frac{(p-1)^*}{p-1}} \\
&= CM^{\frac{n+p-p^2 }{n-(p-1)}}a^{p-1} b^{\frac{n}{n-(p-1)}}.
\end{align*}

To relate $g$ to $g'=-(a+kb)$, we let $q>1$ and use Young's inequality with an $\epsilon$,
\begin{align*}
a^{p-1}b^{\frac{n}{n-(p-1)}} \leq \frac{1}{q} \left( \epsilon a^{p-1}  \right)^q  + \frac{q-1}{q} \left( \frac{b^{\frac{n}{n-(p-1)}}}{\epsilon} \right)^{\frac{q}{q-1}}.
\end{align*}
To match the exponents for $a$ and $b$, we take $q=1 + \frac{n}{(p-1)(n-(p-1))}.$ This gives
\begin{align*}
a^{p-1}b^{\frac{n}{n-(p-1)}} \leq \frac{1}{q}\epsilon^q \left( a^{(p-1)q} + (q-1)\epsilon^{-\frac{q^2}{q-1}}b^{(p-1)q} \right).
\end{align*}
We determine $\epsilon$ from  $$(q-1)\epsilon^{\frac{-q^2 }{q-1}} =k^{(p-1)q}$$
to match the derivative $g'$.
Then
\begin{align*}
a^{p-1}b^{\frac{n}{n-(p-1)}} &\leq \frac{1}{q}\epsilon^q \left( a^{(p-1)q} + \left(  kb \right)^{(p-1)q} \right) \leq  \frac{C\epsilon^q}{q} \left( a+ kb  \right)^{(p-1)q} \\
&=C_2\epsilon^q \left( -g'(\rho) \right)^{(p-1)q}.
\end{align*}
The differential inequality can then be written
\begin{align*}
\frac{g^{p}\epsilon^{-q}}{C_3 M^{\frac{n+p-p^2 }{n-(p-1)}}} \leq \left( -g' \right)^{(p-1)q}.
\end{align*}

Here, we insert the values for $\epsilon$ and $q$. For convenience, put $$\gamma=\gamma(n,p)= np-(p-1)^2>0.$$ Then

\begin{equation}
g' g^{\frac{p-1}{\gamma}-1} \leq - \frac{k^{\frac{n}{\gamma}}}{C_4 M^{\frac{n+p-p^2 }{\gamma}}}.
\label{diffeqnp}
\end{equation}

\section{Proof of the oscillation Theorem}

We are ready to prove Theorem \ref{osclemma} using the differential inequality \eqref{diffeqnp}. From the inequality, it follows that

\begin{align*}
\frac{d}{d\rho} \left( g(\rho) \right)^{\frac{p-1}{\gamma}} =\frac{p-1}{\gamma} g^{\frac{p-1}{\gamma}-1} g' \leq - \frac{k^{\frac{n}{\gamma}}}{C_5 M^{\frac{n+p-p^2 }{\gamma}}}.
\end{align*}

We claim that $g(1)=0$. The above differential inequality can be integrated from $\rho=0$ to $\rho=1$, provided $g(\rho) \neq 0$ for $\rho$ in this range. If $g(\rho)=0$ for some $\rho <1$, the claim is true, since $g' \leq 0$.  If not, we integrate from $0$ to $1$ to get
\begin{align*}
\left( g(1)\right)^{\frac{p-1}{\gamma}}- \left( g(0)\right)^{\frac{p-1}{\gamma}} \leq   -\frac{k^{\frac{n}{\gamma}}}{C_6  M^{\frac{n+p-p^2 }{\gamma}}}.
\end{align*}

By definition, $g(1) \geq 0$, so we want to exclude the case $g(1)>0$. If this is the case, we have
\begin{align*}
0< \left(  g(1)\right)^{\frac{p-1}{\gamma}} \leq \left( g(0)\right)^{\frac{p-1}{\gamma}}-\frac{k^{\frac{n}{\gamma}}}{C_6 M^{\frac{n+p-p^2 }{\gamma}}}.
\end{align*}
We choose $k$ such that the above is zero,
$$k= C_0 M^{\frac{n+p-p^2}{n}} g(0)^{\frac{p-1}{n}} = C_0 \left( \sup_{B_4} u^+ \right)^{\frac{n-p(p-1)}{n}}  \left \{  \int_{B_1} |\nabla u^+|^{p-1} u^+ \, dx  \right \}^{\frac{p-1}{n}}.$$
This gives the contradiction $0<g(1) \leq 0$, so we must have $g(1)=0$, i.e. 

\begin{align*}
\int_{B_1} \left| \nabla (u-k)^+   \right |^{p-1} (u-k)^+ \, dx=0.
\end{align*}

We can rewrite this to $$\int_{B_1}  \left | \nabla \left(  \left( (u-k)^+ \right)^{\frac{p}{p-1}}  \right)      \right|^{p-1} \, dx =0.$$

Then, by the assumptions of Theorem \ref{osclemma} we can use Theorem \ref{poincarezero} to get

\begin{align*}
\int_{B_1} \left( (u-k)^+   \right)^\frac{np}{n-(p-1)} \, dx & = \int_{B_1} \left(   \left( (u-k)^+   \right)^{\frac{p}{p-1}}           \right)^{(p-1)^*} \, dx \\
&\leq C \left\{ \int_{B_1}  \left | \nabla \left(  \left( (u-k)^+ \right)^{\frac{p}{p-1}}  \right)      \right|^{p-1} \, dx \right\}^{\frac{(p-1)^*}{p-1}}=0,
\end{align*}
which means that 
$$(u-k)^+ \equiv 0 \quad \text{in} \, B_1.$$
Hence, $\sup_{B_1} u ^+  \leq k. $
This gives

\begin{align*}
\sup_{B_1} u^+ &\leq k =  C_0 \left( \sup_{B_4} u^+ \right)^{\frac{n-p(p-1)}{n}}  \left \{  \int_{B_1} |\nabla u^+|^{p-1} u^+ \, dx  \right \}^{\frac{p-1}{n}} \\
& \leq C_0 \left( \sup_{B_4} u^+ \right)^{\frac{n-p(p-1)}{n}} \left( \sup_{B_4} u^+ \right)^{\frac{p-1}{n}} \left \{  \int_{B_1} |\nabla u^+|^{p-1}  \, dx  \right \}^{\frac{p-1}{n}} \\ 
&=  C_0 \left( \sup_{B_4} u^+ \right)^{\frac{n-(p-1)^2}{n}}  \left \{  \int_{B_1} |\nabla u^+|^{p-1}  \, dx  \right \}^{\frac{p-1}{n}}.
\end{align*}

By H\"older's inequality we find
\begin{align*}
\int_{B_2} \left |  \nabla u^+  \right |^{p-1} \, dx &= \int_{B_2} \left |  \nabla u^+  \right |^{p-1} \chi_{\{u>0\} \cap B_2} \, dx \\
& \leq \left \{  \int_{B_2} |\nabla u^+|^{p} \, dx \right \}^{\frac{p-1}{p}} \left \{  \int_{B_2} \chi_{\{u>0\} \cap B_2} \, dx \right \}^{\frac{1}{p}} \\
&= \left \{  \int_{B_2} |\nabla u^+|^{p} \, dx \right \}^{\frac{p-1}{p}} \left | \{ u>0 \} \cap B_2 \right |^{\frac{1}{p}},
\end{align*}
which gives 

\begin{align*}
\sup_{B_1} u^+ \leq C_0 \left | \{ u>0 \} \cap B_2 \right |^{\frac{p-1}{np}} \left( \sup_{B_4} u^+ \right)^{\frac{n-(p-1)^2}{n}}\left \{  \int_{B_2} |\nabla u^+|^{p}  \right \}^{\frac{(p-1)^2}{np}}.
\end{align*}
We use the ordinary Caccioppoli inequality, Lemma \ref{cacc1}, with $r=2$. Hence, 

\begin{align*}
\sup_{B_1} u^+ &\leq C_1 \left | \{ u>0 \} \cap B_2 \right |^{\frac{p-1}{np}} \left( \sup_{B_4} u^+ \right)^{\frac{n-(p-1)^2}{n}} \left\{ \int_{B_4} (u^+)^p \, dx   \right\}^{\frac{(p-1)^2}{np}} \\
& \leq C_2 \left | \{ u>0 \} \cap B_2 \right |^{\frac{p-1}{np}} \sup_{B_4} u^+,
\end{align*} 
which proves the Oscillation Theorem.

\textbf{Comment} It has come to my attention that a similar proof has appeared in a  paper by Tiziano Granucci in [Gr].

\end{document}